\documentclass[12pt]{amsart}
\usepackage{amscd}
\usepackage{amsmath}
\usepackage{hyperref}

\usepackage{amssymb}

\numberwithin{equation}{section}

\newtheorem{thm}[equation]{Theorem}
\newtheorem{prop}[equation]{Proposition}

\newtheorem{lemma}[equation]{Lemma}

\newtheorem{question}[equation]{Question}

\theoremstyle{definition}
\newtheorem{rem}[equation]{Remark}

\newtheorem{dfn}[equation]{Definition}
\newtheorem{notation}[equation]{Notation}

\newcommand{\Gal}{\mathop{\mathrm{Gal}}}

\newcommand{\CH}{\mathop{\mathrm{CH}}\nolimits}

\newcommand{\Char}{\mathop{\mathrm{char}}\nolimits}

\newcommand{\Z}{\mathbb{Z}}

\newcommand{\A}{\mathbb{A}}

\newcommand{\Q}{\mathbb{Q}}

\newcommand{\Spec}{\operatorname{Spec}}

\newcommand{\Hom}{\operatorname{Hom}}

\renewcommand{\phi}{\varphi}

\newcommand{\WR}{\mathcal{R}}

\DeclareMathAlphabet{\cat}{OT1}{cmss}{m}{sl}

\title
{Weil Restriction and the Motivic Cycle Class Map}

\keywords
{Motives;
Weil restriction;
Motivic cohomology;
Étale realization;
Chow rings.
{\em Mathematical Subject Classification (2020):}
14F42; 14C25}

\author
{Qi Ge}

\address
{Mathematical \& Statistical Sciences \\
University of Alberta \\
Edmonton
\\
CANADA}

\email
{qge@ualberta.ca}

\author
{Guangzhao Zhu}

\address
{Mathematical \& Statistical Sciences \\
University of Alberta \\
Edmonton
\\
CANADA}

\email
{guangzha@ualberta.ca}

\date
{20 Feb 2026}

\begin{document}

\begin{abstract}
We construct the Weil restriction map for $\ell$-adic cohomology and, 
more generally, for mixed Weil cohomology theories. We study its 
compatibility with the motivic cycle class map and show that these 
constructions admit a natural interpretation in the triangulated 
categories of motives. Using Grothendieck's six-functor formalism, 
we prove that the Weil restriction map arises intrinsically from the 
functorial structures of these categories. This provides a conceptual 
framework for understanding the interaction between Weil restriction, 
motivic cohomology, and realization functors.
\end{abstract}

\maketitle

\tableofcontents

\addtocounter{section}{0}

Cohomology theories of algebraic schemes occupy a central position in modern algebraic geometry. A remarkable example is Deligne's proof of the Weil conjecture, in which he made essential use of deep properties of étale cohomology.

This naturally leads to the question of how different cohomology theories are related. As a bridge between such theories, the étale cycle class map provides an initial link between the Chow theory—a Borel--Moore oriented cohomology theory in the sense of Levine and Morel \cite{MR2286826}—and étale cohomology. Moreover, this relationship can be further clarified within the framework of motivic cohomology—conjectured by Beilinson and Lichtenbaum and later constructed by Voevodsky. In this setting, the Chow theory and the étale cohomology are connected naturally as a comparison map arising from the canonical change of site morphism, relating motivic cohomology to Lichtenbaum cohomology. This comparison map is known as the motivic cycle class map.

In this paper, we study the motivic cycle class map from a descent 
point of view. This is achieved by considering an important construction 
in algebraic geometry—the Weil restriction—which provides a powerful 
tool for descent and has numerous applications in arithmetic and geometry. 

The Weil restriction for algebraic cycles, and for Chow groups, was 
established by Karpenko \cite{MR1809664}. In parallel with Karpenko's work, we construct the Weil restriction map for a large class of cohomology theories with good descent properties, including $\ell$-adic cohomology. 

By working in the triangulated categories that give rise to these 
cohomology theories, we observe that our construction is intrinsic, 
since it relies solely on Grothendieck's six-functor formalism 
in these categories. As a result, this paper provides an attempt 
to understand how concrete constructions and phenomena in algebraic 
geometry are governed by the six-functor formalism.

We organize the paper as follows. Section~\S 1 is mainly preliminary, 
in which we recall the construction and basic properties of the 
Weil restriction of schemes and algebraic cycles. In Section~\S 2, 
we show that the motivic cycle class map, viewed as a comparison map, 
is induced by the étale cycle map. In Section~\S 3, we construct 
the Weil restriction map for $\ell$-adic cohomology and prove a 
compatibility theorem. Finally, in Section~\S 4, we generalize our 
construction and compatibility theorem to mixed Weil cohomology theories.

\subsection*{Acknowledgments}
The second author thanks Prof.~Stefan Gille for introducing the problem of Weil restriction. 
The second author also thanks Prof.~Nikita Karpenko for helpful discussions concerning his work on algebraic cycles.

\subsection*{Notation and conventions}
Throughout this paper, schemes are assumed to be separated and of finite type over a field.  
A \emph{variety} means an integral scheme, i.e., an irreducible and reduced scheme.

An {\em algebra} is a commutative algebra and a {\em module} means a module over a commutative ring.

For an abelian group $\Lambda$, we write $\underline{\Lambda}$ for the associated constant étale sheaf.  

For a scheme $X$, $H^{*}(X,\mathcal{F})$ denotes the cohomology groups of a sheaf $\mathcal{F}$ (on the relevant site), and for a closed subscheme $Z \subset X$, $H^{*}_{Z}(X,\mathcal{F})$ denotes the cohomology groups with support in $Z$.

For two sheaves $\mathcal{F}$ and $\mathcal{G}$ on a given site, we write $\Hom(\mathcal{F},\mathcal{G})$ for the group of global morphisms and $\mathcal{H}om(\mathcal{F},\mathcal{G})$ for the internal Hom sheaf.

\section{Preliminary}
\subsection{Weil restriction of schemes}
We briefly review the basics of Weil restriction of schemes. A detailed discussion can be found in \cite{MR4480537} or \cite{MR1045822}.

Let $L/k$ be a finite Galois field extension, and let $X$ be a scheme over $L$.
The \emph{Weil restriction} of $X$ along the extension $L/k$, denoted by
$\WR_{L/k}(X)$ (or simply $\WR(X)$ when no confusion arises),
is a $k$-scheme (if exists) characterized by the following universal property
\[
\Hom_k\bigl(S,\WR(X)\bigr)
\;\cong\;
\Hom_L\bigl(S \times_k L,\, X\bigr),
\]
for any $k$-scheme $S$. Weil restriction always exists when considering quasi-projective schemes.

Under the same assumptions as above, we may describe $\WR(X)_L$ more explicitly. Let $G:=\Gal(L/k)$ be the Galois group corresponding to 
the field extension. For any element $\sigma \in G$, we denote by $X^{\sigma}$ the {\em conjugate scheme} of $X$, which is obtained by composing the structure morphism of $X$ with $\sigma: \Spec(L)\to \Spec(L)$, the field automorphism induced by $\sigma$. Then there is a canonical isomorphism
\begin{equation}
\WR(X)_L \;\cong\; \prod_{\sigma \in G} X^{\sigma}.
\end{equation}

Weil restriction enjoys several functorial properties, we list some of them in the following. 

\begin{prop}[{\cite[Proposition 1.1]{MR1809664}}]
Let $L/k$ be a finite Galois extension, and assume that the Weil restriction of schemes exists in the following setting. Then:
\begin{itemize}
    \item[(1)] For schemes $X,Y$ over $L$, we have
    \[
        \WR(X \times_{L} Y) \cong \WR(X)\times_{k}\WR(Y).
    \]
    \item[(2)] For the closed embedding of schemes over $L$, $i\colon X\hookrightarrow Y$, the induced morphism
    \[
        \WR(i)\colon \WR(X) \to \WR(Y)
    \]
    is also a closed embedding.
    \item[(3)] If $X$ is a smooth scheme over $L$, then $\WR(X)$ is smooth over $k$.
    \item[(4)] If $X$ is a (quasi-)projective scheme over $L$, then $\WR(X)$ is  also  (quasi-)projective over $k$.
\end{itemize}
\end{prop}

\subsection{Weil restriction of algebraic cycles}
Now we may turn to Weil restriction of algebraic cycles. The original construction is given by Karpenko in~\cite{MR1809664}. We briefly recall the construction for later use.

Recall that the \emph{group of algebraic cycles} on a scheme $X$ is defined as the free abelian group generated by all closed subvarieties of $X$. We denote the group by $Z^{*}(X)$ or simply $Z(X)$.
This group carries a natural grading by codimension, and we therefore write
\[
Z^{*}(X) = \bigoplus_{r=0}^{\dim X} Z^{r}(X),
\]
where $Z^{r}(X)$ denotes the subgroup generated by integral closed subschemes of codimension $r$ in $X$. Moreover, by quotienting out by the relation of \emph{rational equivalence} (see \cite{MR1644323}), we obtain the (integral) \emph{Chow group} of $X$, denoted by $\CH^{*}(X)$ (or simply $\CH(X)$). We write $\CH^{r}(X)$ for the $r$th graded piece of $\CH^{*}(X)$, which is the quotient of $Z^{r}(X)$ by rational equivalence.
We describe functorial properties of algebraic cycles with respect to certain morphisms. 

For a flat morphism $f \colon Y \to X$ of constant relative dimension $d$, there is a pullback homomorphism
\[
f^{*} \colon Z^{r}(X) \to Z^{r}(Y),
\]
and hence $f^{*} \colon Z^{*}(X) \to Z^{*}(Y)$.

For a proper morphism $g \colon X \to Y$, there is a pushforward homomorphism
\[
g_{*} \colon Z^{r}(X) \to Z^{r - (\dim X - \dim Y)}(Y),
\]
and thus $g_{*} \colon Z^{*}(X) \to Z^{* - (\dim X - \dim Y)}(Y)$.

Pullbacks and pushforwards are functorial with respect to composition. Moreover, they respect the relation of rational equivalence; therefore, both operations descend to homomorphisms on Chow groups (see \cite[Chapter~I]{MR1644323}).

Now we can start the construction of Weil restriction of algebraic cycles. Let $L/k$ be a finite Galois extension and let $X$ be a smooth quasi-projective scheme over $L$. The Weil restriction $\WR(X)$ exists and is a smooth quasi-projective scheme over $k$ \cite[Corollary~1.4]{MR1809664}. Let $G=\Gal(L/k)$ be the Galois group.

Consider the natural projection morphism
\[
p \colon \WR(X)_{L} \to \WR(X).
\]
It induces an injective pullback homomorphism
\[
p^{*} \colon Z(\WR(X)) \longrightarrow Z(\WR(X)_{L}).
\]
The key point is that the image of $p^{*}$ is precisely the subgroup of $G$-invariant cycles in $Z(\WR(X)_{L})$ (see \cite[Lemma~2.1]{MR1809664}). In other words, there is an isomorphism
\[
p^{*}\bigl(Z(\WR(X))\bigr) \cong Z(\WR(X)_{L})^{G}.
\]

For any algebraic cycle $\alpha \in Z^{r}(X)$ and any $\sigma \in G$, we denote by $\alpha^{\sigma}$ the cycle on $X^{\sigma}$ obtained by base change via $\sigma$. Then
\[
\prod_{\sigma \in G} \alpha^{\sigma}
\]
defines a cycle in $Z\bigl(\WR(X)_{L}\bigr)$ via the exterior product. This cycle is $G$-invariant, and therefore corresponds to a unique cycle $\WR(\alpha) \in Z^{dr}(\WR(X))$ such that
\[
p^{*}\bigl(\WR(\alpha)\bigr) = \prod_{\sigma \in G} \alpha^{\sigma}.
\]

As a result, we obtain a well-defined map
\[
\WR \colon Z^{r}(X) \longrightarrow Z^{d r}\bigl(\WR(X)\bigr),
\]
where $d = [L:k]$. By taking the direct sum over all graded components, we obtain a map (\textbf{Not} a group homomorphism)
\[
\WR \colon Z^{*}(X) \longrightarrow Z^{d\cdot *}(\WR(X)),
\]
which we call the {\em Weil restriction of algebraic cycles}. The Weil restriction map is compatible with flat pullbacks, proper pushforwards, and both interior and exterior products of cycles \cite[Proposition~2.5]{MR1809664}.

Moreover, the Weil restriction of algebraic cycles is compatible with the relation of rational equivalence \cite[Corollary~3.3]{MR1809664}. Hence, the map descends to the level of Chow groups,
\[
\WR \colon \CH(X) \longrightarrow \CH(\WR(X)).
\]
It remains compatible with pushforwards, pullbacks, products of cycles, and Gysin pullbacks on Chow groups. Finally, it is compatible with the intersection product \cite[Proposition~3.4]{MR1809664}. For a modern proof of compatibility of Weil restriction map with the intersection product, we refer the reader to \cite{MR4949893}, where Fulton's \emph{deformation to the normal cone} is used.

\subsection{Étale cohomology and cycle class map}
In this section, we recall some basics on étale cohomology and the construction of the étale cycle class map. This material is covered in standard expositions on étale cohomology; for example, Milne's \emph{Étale Cohomology} \cite{MR4904233}, which treats primarily the case over an algebraically closed field. In fact, the main statements remain valid over an arbitrary base field. For a more general and comprehensive treatment, we refer the reader to \emph{Cycle, SGA~$4\frac{1}{2}$} \cite{MR3727436}.

We fix a base field $k$. Throughout this section, a sheaf is an étale sheaf; any cohomology group will mean étale cohomology group, defined as sheaf cohomology on the étale site unless otherwise specified. Recall that for an abelian group $A$, we denote by $\underline{A}$ the associated constant étale sheaf.

Recall the following definition of \emph{Tate twist}.
\begin{dfn}\cite[\S2, Chapter V]{MR4904233}
\label{df}

Let $X$ be a scheme over $k$ and $\mathcal{F}$ be a sheaf of $\Z/l\Z$-module with $\ell \ne \mathrm{char}(k)$. 
 Let $\mu_{\ell}$ be the sheaf of $\ell$th roots of unity , we define the $r$th \emph{Tate twist} of $\mathcal{F}$ as follows
\begin{equation*}
  \mathcal{F}(r):=\begin{cases}
\mathcal{F}\otimes\mu_n^{\otimes r} & \text{if } r>0,\\
\mathcal{F} & \text{if } r=0,\\
\mathcal{H}om(\mathcal{F}(-r),\,\mathbf \Z/l\Z) & \text{if } r<0.
\end{cases}
\end{equation*}
\end{dfn}
Now we turn to the construction of the étale cycle class map.

Let $X$ be a smooth scheme over $k$, and let $\Lambda := \Z/\ell\Z$ with $\ell \ne \mathrm{char}(k)$. For a smooth closed subvariety $i\colon Z \hookrightarrow X$ of codimension $r$, we have the Gysin map (see \cite[\S5, Chapter~VI]{MR4904233})
\[
i_{*} \colon H^{0}\!\left(Z, i^{*}\underline{\Lambda}\right) \longrightarrow H^{2r}_{Z}\!\left(X, \underline{\Lambda}(r)\right).
\]
Since $H^{0}(Z, i^{*}\underline{\Lambda}) \cong \Z/\ell\Z$, the image of $1 \in H^{0}(Z, i^{*}\underline{\Lambda})$ defines a class in $H^{2r}_{Z}(X, \underline{\Lambda}(r))$. Composing with the canonical map
\[
H^{2r}_{Z}(X, \underline{\Lambda}(r)) \longrightarrow H^{2r}(X, \underline{\Lambda}(r)),
\]
we obtain a cohomology class in $H^{2r}(X, \underline{\Lambda}(r))$, which we denote by $cl_{X}^{r}(Z)$.

For a possibly singular closed subvariety $Z' \subset X$, one can still define such a class using \emph{semi-purity} \cite[Proposition~2.2.6]{MR3727436}, which yields an identification
\[
H^{2r}_{Z'}(X, \underline{\Lambda}(r))
\;\cong\;
H^{2r}_{Z' \setminus Sing(Z')}\!\bigl(X \setminus Sing(Z'), \underline{\Lambda}(r)\bigr),
\]
where $Sing(Z')$ denotes the singular locus of $Z'$. This reduces the construction to the case of a closed embedding of smooth schemes
\[
i'\colon Z' \setminus Sing(Z') \hookrightarrow X \setminus Sing(Z').
\]
Finally, by extending the construction linearly, we obtain the {\em$r$th  étale cycle map}
\begin{equation}
    \label{eqe1}
    cl_{X}^{r} \colon Z^{r}(X) \longrightarrow H^{2r}\!\left(X, \underline{\Lambda}(r)\right).
\end{equation}
By taking the direct sum over all codimensions $r \ge 0$, we obtain the {\em total étale cycle map}
\begin{equation}
    \label{eqe2}
cl_{X} \colon Z^{*}(X) \longrightarrow H^{2*}(X):=\oplus_{i=0} ^{dim(X)}H^{2i}\!\left(X, \underline{\Lambda}(i)\right).
\end{equation}

The étale cycle map enjoys good compatibility with operations on algebraic cycles. We record some of these properties in the following theorem.

\begin{thm}[{\cite[\S9, Chapter~VI]{MR4904233}}]
\label{t1}
Let $X$ be a smooth scheme over $k$. Under the above assumptions, the following hold:
\begin{itemize}
    \item[(1)] For any finite étale morphism $f\colon Y \to X$, the étale cycle map is compatible with the pullback of algebraic cycles along $f$.
    \item[(2)] The étale cycle map is compatible with exterior products of algebraic cycles via the Künneth map in cohomology.
    \item[(3)] For two algebraic cycles that intersect transversely, the étale cycle map is compatible with the intersection product via the cup product in cohomology.
\end{itemize}
\end{thm}
\begin{rem}
    In Theorem~\ref{t1}(1), the hypothesis can in fact be weakened to the condition that the codimension of the inverse image of any subvariety of $X$ under a morphism $f \colon Y \to X$ is lower semi-continuous on the fibers. See \cite{MR3727436}.
\end{rem}

In certain situations, for example when the étale cohomology groups 
have coefficients in a field of characteristic zero, the étale cycle 
map extends (with the same image) to a group homomorphism from the Chow group to the étale cohomology group with good functorial properties. 

In this case, we refer to such a map as the {\em étale cycle class map}.

\section{Motivic cycle class map}
\subsection{Comparison map}
Fix a field $k$. Let $\textbf{Sm}_{k}$ denote the category of smooth schemes over $k$. For each integer $q$, there is a complex (the Voevodsky--Suslin complex) of Nisnevich sheaves with transfers, denoted $\Z(q)$, on $\textbf{Sm}_{k}$ (equipped with the Nisnevich topology). The \emph{motivic cohomology} of a smooth scheme $X$ is defined through the hypercohomology $
H^{p,q}(X,\Z) := \mathbb{H}^{p}_{\mathrm{Nis}}\!\bigl(X, \Z(q)\bigr)$. (Indeed, $\Z(q)$ is also a complex of Zariski sheaves. For a detailed discussion, see \cite[Part 3]{MR2242284})

Let $\pi \colon X_{\acute{e}t} \to X_{\mathrm{Nis}}$ be the canonical morphism of sites from the étale site to the Nisnevich site. Using the unit of adjunction
$\mathrm{id} \longrightarrow R\pi_{*}\pi^{*}$,
we obtain a canonical morphism
\begin{equation}
\label{eq}
\Z(q) \longrightarrow R\pi_{*}\pi^{*}\Z(q).
\end{equation}
Since $\pi^{*}\Z(q)$ is the étale sheafification of $\Z(q)$, passing to hypercohomology yields the comparison map
$
\mathbb{H}^{p}_{\mathrm{Nis}}\!\bigl(X,\Z(q)\bigr)
\longrightarrow
\mathbb{H}^{p}_{\acute{e}t}\!\bigl(X,\Z(q)\bigr)
$. The left-hand side is motivic cohomology, while the right-hand side is known as \emph{Lichtenbaum cohomology}, which we denote by $H_{L}^{p,q}(X,\Z)$. Thus we obtain the comparison map $ H^{p,q}(X,\Z) \longrightarrow H_{L}^{p,q}(X,\Z)$.

For any abelian group $A$, define $A(q):=\Z(q)\otimes_{\Z} A$. This gives motivic and Lichtenbaum cohomology with coefficients in $A$, and hence a comparison map in general:
\[
H^{p,q}(X,A) \longrightarrow H_{L}^{p,q}(X,A).
\]

Now consider the case for the group $\Lambda:=\Z/\ell\Z$ with $\ell \ne \mathrm{char}(k)$. There is a quasi-isomorphism in the category of bounded above complexes of étale sheaves
\[
\mathrm{Cone}\bigl(\Z(1)_{\acute{e}t} \xrightarrow{\ \times \ell\ } \Z(1)_{\acute{e}t}\bigr)
\simeq
\mu_{\ell}[0],
\]
 for the étale sheaf of $\ell$th roots of unity $\mu_{\ell}$. It follows that
\[
H_{L}^{p,q}(X,\Z/\ell\Z) \cong H^{p}_{\acute{e}t}\!\left(X,\underline{\Lambda}(q)\right).
\]
where $\underline{\Lambda}(q)$ denotes the $q$th Tate twist for $\underline{\Lambda}:=\underline{\Z/l\Z}$ (See Definition \ref{df}).

When $p=2q$, the comparison theorem for motivic cohomology in weight $q$ implies that
$H^{2q,q}(X,A) \cong \CH^{q}(X)\otimes A$ for an abelian group $A$.
Consequently, we obtain the \emph{motivic cycle class map}
\begin{equation}
\label{eqn}
cl_{\mathrm{mot}}^{q} \colon \CH^{q}(X)\otimes \Lambda \longrightarrow H^{2q}_{\acute{e}t}\!\left(X,\underline{\Lambda}(q)\right).
\end{equation} where $\Lambda=\Z/l\Z$ with $l\ne \Char(k)$.

A natural question is the following.

\begin{question}
Let $X$ be a smooth scheme over $k$. Is the motivic cycle class map 
$cl_{\mathrm{mot}}^{q}$ 
induced by the étale cycle map considered in \S1.3?
\end{question}

We give a positive answer to this question by carefully examining the construction of the comparison morphism from motivic cohomology to Lichtenbaum cohomology due to Geisser and Levine \cite{MR1807268}. Although this result has been known to experts, a detailed reference is hard to find, so we include a proof following their method.

\subsection{Naturality of the motivic cycle class map}
Let $X$ be a smooth scheme over $k$, and let $\Lambda := \Z/\ell\Z$ with $\ell \ne \mathrm{char}(k)$.  
For any integer $q \ge 0$, let $z^{q}(X,*)$ denote Bloch's cycle complex (of Zariski sheaves) computing higher Chow groups \cite[Part~5]{MR2242284}. It is known that, after a suitable reindexing, the complex $z^{q}(X,*)$ is quasi-isomorphic to the Voevodsky--Suslin complex $\Z(q)$ (see \cite[Section~2.5]{MR1807268}). We denote the reindexed complex by $Z^{q}(X,*)$. In particular, its degree-$2q$ term satisfies
$Z^{q}(X,2q) = Z^{q}(X)$ , the group of codimension-$q$ algebraic cycles on $X$.

Let $\pi: X_{\text{ét}} \to X_{Zar}$ be the canonical morphism from the étale site to the Zariski site. Then the construction of Geisser and Levine yields a natural map (\cite[Equation 3.21]{MR1807268}) \begin{equation}
\label{eq2}
    Z^{q}(X,*)\otimes \Lambda \to R\pi_{*}\pi^{*}(\underline{\Lambda}(q))\end{equation}
\begin{prop}\cite[cf.\ Theorem~1.1]{MR1807268}
\label{prop1}
Under the same assumptions as above, the map on hypercohomology induced by \eqref{eq2} coincides with the motivic cycle class map considered in \S2.1 (i.e., \eqref{eqn}).
\end{prop}
\begin{proof}
Since $Z^{q}(X,*)$ is quasi-isomorphic to $\Z(q)$ as complexes of Zariski sheaves, and $\pi^{*}\Z(q)\otimes^{\mathbf L}\underline{\Lambda}$ is quasi-isomorphic to $\pi^{*}(\underline{\Lambda}(q))$ as complexes of étale sheaves, we have that the map \eqref{eq2} is induced by the unit of adjunction
\[
\mathrm{id} \longrightarrow R\pi_{*}\pi^{*}
\]
in the derived category of sheaves. Comparing with \eqref{eq}, the uniqueness of the unit of adjunction shows that, after passing to hypercohomology, the map
\[
H^{p,q}(X,\Lambda)
\longrightarrow
H^{p,q}_{L}(X,\Lambda)
= H^{p}_{\acute{e}t}\!\left(X,\Lambda(q)\right)
\]
induced by \eqref{eq2} coincides with the motivic cycle class map considered in \S2.1
\end{proof}

\subsection{Motivic versus étale cycle class map}
We conclude this section by proving the following theorem. We remind the reader that the proof follows directly from the construction of the motivic cycle class map given by Geisser and Levine \cite{MR1807268}.
\begin{thm}
\label{com}
    Let $X$ be a smooth scheme over $k$, and let $\Lambda := \Z/\ell\Z$ with $\ell \ne \mathrm{char}(k)$. Then the motivic cycle class map \eqref{eqn} is obtained from the étale cycle map by tensoring with $\Lambda$ and passing to the quotient by rational equivalence relation.
    \end{thm}
\begin{proof}
To prove the theorem, it suffices to verify the statement on each graded component of the Chow group of $X$. Recall that $Z^{q}(X,*)$ is the reindexed Bloch cycle complex, with $Z^{q}(X)\otimes\Lambda=Z^{q}(X,2q)\otimes \Lambda$. By Proposition~\ref{prop1}, we reduce to determining the image of the natural map (cf. \eqref{eq2})
\begin{equation}
\label{e1}
    \eta: Z^{q}(X,*)\otimes \Lambda \to R\pi_{*}\pi^{*}(\underline{\Lambda}(q))  
\end{equation}
As a complex of étale sheaves, $\pi^{*}\!\bigl(\underline{\Lambda}(q)\bigr)$ is quasi-isomorphic to $G^{*}\!\bigl(X,\underline{\Lambda}(q)\bigr)$, the Godement resolution of $\underline{\Lambda}(q)$ (\cite[Example 3.3]{MR1807268}). Hence the map $\eta$ is reduced to \begin{equation}
    Z(X,*)\otimes \Lambda \to G^{*}(X,\underline{\Lambda}(q)).\end{equation} (See \cite[\S3.7]{MR1807268})
which yields a map $cl^{q}: Z^{q}(X)\otimes \Lambda \to H^{2q}(G^{*}(X,\underline{\Lambda}(q)))$.

For any closed subvariety $W \subset X$ of codimension $q$, let 
$Z^{q}_{W}(X)$ denote the group of codimension-$q$ cycles on $X$ supported in $W$, and let 
$G^{*}_{W}\!\bigl(X,\underline{\Lambda}(q)\bigr)$ denote the Godement resolution of $\underline{\Lambda}(q)$ with support in $W$. 
Clearly, $W$ defines an element of $Z^{q}_{W}(X)$, and we have a natural inclusion $Z^{q}_{W}(X) \subset Z^{q}(X)$.

From \cite[Lemma~3.8]{MR1807268}, we obtain the following commutative diagram:
\[
\begin{CD}
Z^{q}_{W}(X)\otimes \Lambda 
  @>{}>>
H^{2q}\!\bigl(G^{*}_{W}(X),\underline{\Lambda}(q)\bigr) \\
@VVV @VVV \\
Z^{q}(X)\otimes \Lambda 
  @>{cl^{q} }>>
H^{2q}\!\bigl(G^{*}(X),\underline{\Lambda}(q)\bigr),
\end{CD}
\]
where the vertical arrows are the natural maps.

Viewing $W$ as a cycle in both $Z^{q}_{W}(X)$ and $Z^{q}(X)$, note that the upper part of the digram is nothing more than the étale cycle map modulo $\Lambda$, we conclude that
\[
cl^{q}(W) = \bigl(cl^{q}_{X}\otimes \Lambda\bigr)(W),
\]
which completes the proof.
    
\end{proof}

\section{Weil restriction and the motivic cycle class map}
\subsection{Weil restriction on $\ell$-adic cohomology}
In this section, we construct the Weil restriction on $\ell$-adic cohomology. Unless otherwise specified, all cohomology groups will refer to étale cohomology groups.

We begin by recalling some basic notions. Fix a field $k$ and let $X$ be a scheme over $k$. Let $\ell$ be a prime number with $\ell \ne \mathrm{char}(k)$, and set
\[
\Lambda_{n} := \Z/\ell^{n}\Z.
\]
The {\em$\ell$-adic cohomology of $X$ with coefficients in $\Q_{\ell}$} is defined by
\[
H^{i}\!\left(X,\Q_{\ell}(r)\right)
:= \left(\varprojlim_{n} H^{i}\!\left(X,\underline{\Lambda_{n}}(r)\right)\right)\otimes_{\Z_{\ell}}\Q_{\ell},
\]
where the inverse limit is taken with respect to the transition maps
\[
\underline{\Lambda_{n+1}}(r) \longrightarrow \underline{\Lambda_{n}}(r)
\]
induced by the natural map $\Z/l^{n+1}\Z \to \Z/l^{n}\Z$.

\begin{notation}
To simplify notation, we will refer to $\ell$-adic cohomology with coefficients in $\Q_{\ell}$ simply as $\ell$-adic cohomology. Also, we will always assume that $\ell \ne \Char k$. This will be used in the rest of the paper.
\end{notation}

There are two reasons to consider $\ell$-adic cohomology. First, the étale cycle class maps \eqref{eqe1} and \eqref{eqe2} do not, in general, extend to Chow groups without additional hypotheses on the dimension of the ambient scheme. In contrast, for $\ell$-adic cohomology, such an extension is always possible \cite[Remark 10.8, \S10, Chapter VI]{MR4904233}.
Second, $\ell$-adic cohomology has good descent properties, which are described in the following lemma.

\begin{lemma}
\label{inv}
 Let $L/k$ be a finite Galois extension with Galois group $G$, and let $X$ be a scheme over $k$. Then, for all integers $i$ and $r$, there is a canonical isomorphism
\[
H^{i}\!\left(X_{L},\Q_{\ell}(r)\right)^{G}
\;\cong\;
H^{i}\!\left(X,\Q_{\ell}(r)\right).
\]
\end{lemma}
\begin{proof}
Note that the projection morphism $X_{L}\to X$ is a finite Galois covering morphism (In the sense of  \cite[Remark 5.4, \S5, Chapter I]{MR4904233}). Hence we have the Hoschild-Serre spectral sequence\[E_{2}^{p,q}=H^{p}(G,H^{q}(X_{L}, \Q_{l}(r)))\Rightarrow H^{p+q}(X,\Q_{l}(r)).\] (\cite[Theorem 2.20,\S2, Chapter III]{MR4904233}).

Since $H^q(X_{L},\Q_{l}(r))$ is a vector space over $\Q_{l}$, which is uniquely $n$-divisible, we have that $E^{p,q}_{2}=0$ except for the term $p=0$ (\cite[Corollary 3.3.9, Chapter III]{MR3727161}). Hence we have \[H^{q}(X,\Q_{l}(r))\cong H^{0}(G,H^{q}(X_{L},\Q_{l}(r)))=H^{q}(X_{L},\Q_{l}(r))^{G}\] This finishes our proof.
\end{proof}

We now turn to the construction of the Weil restriction on $\ell$-adic cohomology. Let $L/k$ be a finite Galois extension with Galois group $G$, and let $X$ be a smooth quasi-projective scheme over $L$. We assume that $[L:k]=n$. Still, we denote by $\WR(X)$ the Weil restriction of $X$ along $L/k$.

For any cohomology class $f \in H^{i}\!\left(X,\Q_{\ell}(r)\right)$ and any $\sigma \in G$, we denote by $f^{\sigma}$ the corresponding class in
\[
H^{i}\!\left(X^{\sigma},\Q_{\ell}(r)\right),
\]
where $X^{\sigma}$ is the conjugate scheme as defined in \S~1.1.

Since $X$ is quasi-projective, one can describe $f^{\sigma}$ explicitly using Čech cohomology (see \cite[Theorem~2.17, \S2, Chapter~III]{MR4904233}). Let $U \to X$ be an étale covering morphism, and set
\[
U^{[i+1]} := \underbrace{U \times_X U \times_X \cdots \times_X U}_{(i+1)\text{ times}}.
\]
Then a class $f \in H^{i}\!\left(X,\Q_{\ell}(r)\right)$ can be represented by a Čech $i$-cocycle, that is, by a section
\[
f_i \in \Gamma\!\left(U^{[i+1]}, \Q_{\ell}(r)\right)
\]
satisfying the usual Čech cocycle condition.
The section $f_i$ should be viewed as an étale-locally constant section with values in the sheaf $\Q_{\ell}(r)$. Then $f_i^{\sigma}$ corresponds to the same Čech cocycle, but with its values in $\Q_{\ell}(r)$ twisted by the natural Galois action of $\sigma$ on the Tate twist.

Consider the following Künneth product homomorphism (\cite[\S8, Chapter VI]{MR4904233}): 

\begin{equation}
    \phi: \prod_{\sigma \in G}H^{i}(X^{\sigma},\Q_{l}(r)) \to H^{ni}(\prod_{\sigma \in G}X^{\sigma},\Q_{l}(nr))
\end{equation}

Now let $f \in H^{i}\!\left(X,\Q_{\ell}(r)\right)$ be a cohomology class. Define
\[
f' := \phi\!\left(\prod_{\sigma \in G} f^{\sigma}\right)
\in H^{ni}\!\left(\prod_{\sigma \in G} X^{\sigma}, \Q_{\ell}(nr)\right),
\] Using Čech representatives, one checks that $f'$ is $G$-invariant. Since
\[
\WR(X)_{L} \cong \prod_{\sigma \in G} X^{\sigma},
\]
Lemma~\ref{inv} implies that $f'$ descends to a unique class in
\[
H^{ni}\!\left(\WR(X), \Q_{\ell}(nr)\right),
\]
which we denote by $\WR(f)$.

This construction defines a map
\[
\WR \colon H^{i}\!\left(X,\Q_{\ell}(r)\right)
\longrightarrow
H^{ni}\!\left(\WR(X),\Q_{\ell}(nr)\right),
\]
called the Weil restriction on $\ell$-adic cohomology.

\subsection{Compatibility}
In this section, we prove the compatibility between the Weil restriction on Chow groups and that on $\ell$-adic cohomology, formulated via motivic cohomology and the motivic cycle class map. Note that the étale cycle map \eqref{eqe1} (also for \eqref{eqe2}) is compatible with the passage to $\ell$-adic cohomology. Hence they extend to maps
\begin{equation}
\label{elq}
cl^{q}_{-} \colon Z^{q}(-) \longrightarrow H^{2q}\!\left(-,\Q_{\ell}(q)\right)
\end{equation}
for any smooth $k$-scheme.
Moreover, the map \eqref{elq} is compatible with rational equivalence. Hence it descends to a homomorphism
\[
cl_{-}^{'\,q} \colon \CH^{q}(-) \longrightarrow H^{2q}\!\left(-,\Q_{\ell}(q)\right).
\]
with the property that  $cl_{-}^{q}$ and $cl_{-}^{'\,q}$ have the same image (\cite[\S10, Chapter VI]{MR4904233}).

\begin{notation}
    We will, by abuse of notation, use $cl^{q}_{-}$ to denote the $q$th étale cycle (class) map defined on both algebraic cycles and Chow groups.
\end{notation}
Our main result in this section is the following theorem.

\begin{thm}
\label{tt}
Let $L/k$ be a finite Galois extension of degree $n$ and $X$ be a smooth quasi-projective scheme over $L$. For any integer $q$ with $0 \le q \le \dim X$, the following diagram is commutative:
\[
\begin{CD}
H^{2q,q}(X,\Q_{\ell}) @>{cl_{\mathrm{mot}}^{q}}>> H^{2q,q}_{L}(X,\Q_{\ell}) \\
@V{\WR}VV @VV{\WR}V \\
H^{2nq,nq}\!\left(\WR(X),\Q_{\ell}\right) 
  @>{cl_{\mathrm{mot}}^{nq}}>> 
H^{2nq,nq}_{L}\!\left(\WR(X),\Q_{\ell}\right),
\end{CD}
\]
where the left vertical map is induced by the Weil restriction on Chow groups, and the right vertical map is the Weil restriction on $\ell$-adic cohomology via the standard identifications in motivic cohomologies.
\end{thm}
Under the assumptions of Theorem~\ref{tt}, we may apply Theorem~\ref{com} together with the identifications
\[
H^{2*,*}(X,\Q_{\ell}) \cong \CH^{*}(X)\otimes \Q_{\ell}
\quad\text{and}\quad
H^{2*,*}_{L}(X,\Q_{\ell}) \cong H^{2*}\!\left(X,\Q_{\ell}(*)\right)
\]
to reduce the proof to that of the following proposition.
\begin{prop}
Under the assumptions of Theorem~\ref{tt}, for any $q\in [0,dim(X)]$, we have the following commutative diagram
\[
\begin{CD}
Z^{q}(X) @>{cl^{q}_{X}}>> H^{2q}\!\left(X,\mathbb{Q}_{\ell}(q)\right) \\
@V{\WR}VV @VV{\WR}V \\
Z^{nq}\!\left(\WR(X)\right) @>{cl^{nq}_{\WR(X)}}>> 
H^{2nq}\!\left(\WR(X),\mathbb{Q}_{\ell}(nq)\right)
\end{CD}
\]
\end{prop}

\begin{proof}
Let $G=\Gal(L/k)$. Using Lemma \ref{inv} and \cite[Lemma 2.1]{MR1809664}, we have that $Z^{nq}(X)\cong Z^{nq}(\prod_{\sigma}X^{\sigma})^{G}$ and $H^{2nq}(\WR(X),\Q_{l}(nq))=H^{2nq}(\prod_{\sigma\in G}X^{\sigma},\Q_{l}(nq))^{G}$. Moreover, from the construction of the Weil restriction map, we have that $\WR: Z^{q}(X)\to Z^{nq}(\WR(X))$ factors through the map\[\phi: Z^{q}(X)\to Z^{nq}(\prod_{\sigma\in G}X^{\sigma}); \alpha \mapsto \prod_{\sigma\in G}\alpha^{\sigma}\] while the map $\WR: H^{2q}(X,\Q_{l}(q))\to H^{2nq}(\WR(X),\Q_{l}(nq))$ factors through the map\[\psi: H^{2q}(X,\Q_{l}(q)) \to H^{2nq}(\prod_{\sigma\in G}X^{\sigma},\Q(nq));f \mapsto \prod_{\sigma}f^{\sigma}.\]

As a result, to proves the commutativity of the diagram in the proposition, it suffices to prove the the commutativity of the following diagram
\begin{equation}
 \label{1}   
\begin{CD}
Z^{q}(X) @>{cl_{X}^{q}}>> H^{2q}(X,\mathbb{Q}_{\ell}(q)) \\
@V{\phi}VV @V{\psi}VV \\
Z^{nq}\!\left(\prod_{\sigma\in G} X^{\sigma}\right)
@>{cl_{\prod X^{\sigma}}^{nq}}>>
H^{2nq}\!\left(\WR(X),\mathbb{Q}_{\ell}(nq)\right)
\end{CD}
\end{equation}
Now for an algebraic cycle $\alpha \in Z^{q}(X)$, the upper part of the diagram \ref{1} gives a cycle $\prod_{\sigma \in G}(cl_{X}^{q}(\alpha))^{\sigma}$ while the lower part gives a cycle $cl_{\prod X^{\sigma}}^{nq}(\prod_{\sigma \in G} \alpha^{\sigma})$. Since the étale cycle map is compatible with exterior product, we have that \[cl_{\prod X^{\sigma}}^{nq}(\prod_{\sigma \in G} \alpha^{\sigma})=\prod_{\sigma}cl_{X^{\sigma}}(\alpha^{\sigma}).\]
Now the statement follows from the compatibility of étale cycle class map with flat pullback (Theorem \ref{t1}) applied to the canonical isomorphism $\sigma: X^{\sigma}\to X$.
\end{proof}

\section{A Generalization to mixed Weil theory}
In this section, we give a conceptual explanation of the preceding construction. As a consequence, we generalize our results to a broad class of cohomology theories satisfying properties analogous to those of étale cohomology, namely the mixed Weil theories introduced and studied by F. Déglise. The key tool is Grothendieck’s six-functor formalism in the setting of $\Q$-linear stable model categories. 

\subsection{Mixed Weil theory}
We first recall some basic definitions and properties of mixed Weil theories. For a comprehensive discussion, we refer the reader to \cite{MR2900540} and \cite{MR3971240}. 

We fix a perfect field $k$ and let $K$ be a field of characteristic zero. 
For a scheme $S$, let $\textbf{Sm}_{S}$ be the category of smooth $S$-schemes, and let
$S_{\mathrm{Nis}}$ denote the Nisnevich site on $\textbf{Sm}_S$. 

We write $D^{\mathrm{eff}}_{\mathbb{A}^1}(S,K)$ for the \emph{triangulated category of real effective motives} over $S$. Briefly, the category $D^{\mathrm{eff}}_{\mathbb{A}^{1}}(S,K)$ is obtained from the derived
category of complexes of Nisnevich sheaves of $K$-modules on $S_{\mathrm{Nis}}$
by localizing the homotopy relation $X\times \mathbb{A}^{1} \sim  X$. For any smooth $S$-scheme $X$, let $K(X)$ denote the {\em motive} of $X$ in
$D^{\mathrm{eff}}_{\mathbb{A}^1}(S,K)$ associated to the Nisnevich presheaf
$U\mapsto K[\Hom_S(U,X)]$ (viewed as a complex concentrated in degree $0$).

We denote by $K(1)$ the {\em Tate motive} (See \cite[\S 1.3]{MR2900540}) in $D^{\mathrm{eff}}_{\mathbb{A}^1}(S,K)$. Then $K(n)$ denotes the $n$-fold tensor product of the Tate motive for any non-negative integer $n$. Note that the category $D^{\mathrm{eff}}_{\mathbb{A}^{1}}(S,K)$ 
carries a natural symmetric monoidal structure induced by the derived 
tensor product of complexes of sheaves. For any integer $n \ge 0$, 
we define the {\em $n$th twist} of $K(X)$ by 
\begin{equation}
\label{111}
    K(X)(n) := K(X) \otimes K(n),
\end{equation}

A \emph{mixed Weil theory} over $S$ with coefficients in $K$ is a complex of Nisnevich presheaves 
of graded $K$-algebras on $S_{\mathrm{Nis}}$ which satisfies the following properties: $\mathbb{A}^{1}$-locality, 
Nisnevich excision, the dimension axiom, stability, and the Künneth formula 
(see \cite[\S 2.1.2]{MR2900540}).

For any smooth $S$-scheme $X$ and a mixed Weil theory $E$, we define the ($q$th) {\em mixed Weil cohomology group} of $X$ with respect to $E$ by 
\begin{equation}
\Hom_{D^{\mathrm{eff}}_{\mathbb{A}^{1}}(S,K)}
\bigl(K(X),E[q]\bigr)\end{equation} for any integer $q$.

By abuse of notation, we also write $E$ for the corresponding object in 
$D^{\mathrm{eff}}_{\mathbb{A}^{1}}(S,K)$ whenever the context makes clear 
which underlying category is being considered. Using the defining properties of a mixed Weil theory $E$, for any smooth $S$-scheme $X$ and any integer $q$, we get the following canonical isomorphism :\begin{equation}
    \mathbb{H}_{\mathrm{Nis}}^{q}(E(X))\cong \Hom_{D^{\mathrm{eff}}_{\mathbb{A}^{1}}(S,K)}
\bigl(K(X),E[q]\bigr)\end{equation}
where the left-hand side is the $i$th Nisnevich hypercohomology group of $X$ with respect to $E$.

We may consider the Tate twists of a mixed Weil cohomology group. For a mixed Weil theory $E$ and a smooth $S$-scheme $X$, the {\em $p$th Tate twist of the mixed Weil cohomology group} is defined by 
\begin{equation}
\label{222}
H^{*}(X,E)(p) := H^{*}\bigl(X,E(p)\bigr)\end{equation}

In the above discussion, we restricted ourselves to non-negative integers 
when considering Tate twists. This limitation can be removed by passing to 
a larger category into which 
$D^{\mathrm{eff}}_{\mathbb{A}^{1}}(S,K)$ embeds \textbf{fully faithfully}. 
By stabilizing the category of real effective motives with respect to the 
Tate motive, we obtain the desired category, namely the 
\emph{triangulated category of real mixed motives}, denoted by $D_{\mathbb{A}^{1}}(S,K)$.

A key observation made by F.Déglise is that, any mixed Weil theory $E$ canonically gives rise to a commutative ring spectrum, which we denote by $\mathcal{E}$. With the same notations, for a smooth $S$-scheme $X$ and integers $p,q \in \mathbb{Z}$, 
we define the {\em $q$-th cohomology group of $X$ with twist $p$ 
and coefficients in $\mathcal{E}$} by
\[
H^{q}(X,\mathcal{E}(p))
:=
\Hom_{D_{\mathbb{A}^{1}}(S,K)}
\bigl(K(X),\mathcal{E}(p)[q]\bigr),
\]
where $\underline{K}(X)$  denotes the stabilized 
motive of $X$ in $D_{\mathbb{A}^{1}}(S,K)$. Moreover, we have $H^{q}(X,E)=H^{q}(X,\mathcal{E})$.

Given a mixed Weil theory $E$ and its associated commutative ring spectrum $\mathcal{E}$, a corresponding symmetric monoidal category of $\mathcal{E}$-modules is built up from $D_{\mathbb{A}^{1}}(S,K)$. Moreover, it is endowed with the stable $\A^{1}$-weak equivalences (See \cite[\S1.5.1 \& \S2.1.5]{MR2900540}). We denote this category by $D_{\A^{1}}(S,\mathcal{E})$. In general, we have the following isomorphism relates $D_{\A^{1}}(S,\mathcal{E})$ with $D^{\mathrm{eff}}_{\mathbb{A}^{1}}(S,K)$:
\[\Hom_{D^{\mathrm{eff}}_{\mathbb{A}^{1}}(S,K)}(K(X),E[q]\cong \Hom_{D_{\A^{1}}(S,\mathcal{E})}(\underline{K}(X),\mathcal{E}[q])\] for any integer $q$ and any smooth $S$-scheme $X$.

We summarize the main structural properties of 
$D_{\mathbb{A}^{1}}(S,\mathcal{E})$ proven by F.Déglise in the following theorem.

\begin{thm}[{\cite[\S 2]{MR2900540}}]
\label{tt7}
Let $S$ be a scheme and $K$ a field of characteristic zero. We let $E$ be a mixed Weil theory and $\mathcal{E}$ its associated commutative ring spectrum.

\begin{itemize}

\item[(1)] The triangulated category $D_{\mathbb{A}^{1}}(S,\mathcal{E})$ admits 
Grothendieck's six-functor formalism in the sense of 
\cite[\S Introduction~A.5.1]{MR3971240}. 
Moreover, it satisfies the axioms of absolute purity and duality 
(\cite[\S 2]{MR2900540}).

\item[(2)] Let $k$ be a perfect field and set $S=\Spec(k)$. 
Let $DM_{gm}(k,\Q)$ denote Voevodsky's triangulated category of geometric motives with rational coefficients. 
There exists a symmetric monoidal triangulated functor
\[
R_\mathcal{E} : DM_{gm}(k,\Q) \longrightarrow D_{\mathbb{A}^{1}}(S,\mathcal{E}),
\]
called the realization functor associated to $\mathcal{E}$. 
For any smooth $k$-scheme $X$, this functor induces the {\em motivic cycle class map}
\[
cl_{X}^{q}:H^{q}(X,\Q(p)) \longrightarrow H^{q}(X,\mathcal{E}(p)),
\]
where $H^{q}(X,\Q(p))$ denotes the motivic cohomology groups considered in~\S2. 
\end{itemize}
\end{thm}

\subsection{Galois descent}
The main goal of this subsection is to prove analogues of Lemma~\ref{inv} 
for motivic cohomology with rational coefficients and for mixed Weil 
cohomology. 

The main difficulty is that, unlike in the case of $\ell$-adic cohomology, 
there is no immediate descent statement arising from a Hochschild--Serre 
spectral sequence since a finite Galois covering morphism is 
generally far from being a Nisnevich covering.

Fortunately, the full six-functor formalism available in the triangulated 
categories defining these cohomology theories provides the necessary 
tool to establish the desired Galois descent results. More importantly, 
this shows that our construction in \S3.1 is intrinsic.

We fix a perfect field $k$ and let $K$ be a field of characteristic zero. Let $E$ be a mixed Weil theory and $\mathcal{E}$ its associated commutative ring spectrum.

We will adopt the following notation.
\begin{notation}
Let $DM_{gm}(k,\Q)$ denote Voevodsky's triangulated category of 
geometric motives with rational coefficients, and let 
$D_{\mathbb{A}^{1}}(k,\mathcal{E})$ denote the triangulated category 
of $\mathcal{E}$-modules over $S=\Spec(k)$ with coefficients in $K$, 
as considered in \S 4.1.

For a smooth $k$-scheme $X$, we denote by $M(X)$ its associated motive 
in either $DM_{gm}(k,\Q)$ or $D_{\mathbb{A}^{1}}(k,\mathcal{E})$; 
the ambient category will be clear from the context.

Similarly, we write $H^{i}(X,-)$ for the cohomology groups computed 
in either $DM_{gm}(k,\Q)$ or $D_{\mathbb{A}^{1}}(k,\mathcal{E})$.
\end{notation}

Let $L/k$ be a finite Galois extension with Galois group $G:=\Gal(L/k)$. 
Let $X$ be a smooth $k$-scheme, and denote by 
\[
f : X_L := X \times_{\Spec(k)} \Spec(L) \longrightarrow X
\]
the finite Galois covering morphism.

\begin{thm}
\label{t5}
Under the above assumptions, for all integers $p,q$, there are canonical isomorphisms
\[
H^{p}(X_{L},\Q(q))^{G}
\;\cong\;
H^{p}(X,\Q(q)),
\]
and
\[
H^{p}(X_{L},\mathcal{E}(q))^{G}
\;\cong\;
H^{p}(X,\mathcal{E}(q)).
\]
Here the $G$-action on $H^{p}(X_{L},-)$ is induced by the natural action 
of $G$ on $X_{L}$ via the automorphisms
\[
g : X_{L} \longrightarrow X_{L},
\qquad g \in G.
\]
\end{thm}
Before proving Theorem~\ref{t5}, we establish the following lemma 
concerning the transfer homomorphism in the same setting.

\begin{lemma}
\label{trr}
Let $L/k$ be a finite Galois extension with Galois group $G$. For a smooth $k$-scheme $X$, we let $f : X_L \longrightarrow X $ 
be the finite Galois covering morphism. Then for any cohomology class $\alpha \in H^{i}(X,-)$, we have
\[
f_{*} f^{*}(\alpha) = |G| \cdot \alpha.
\]
Here $f^{*}$ denotes the pullback and $f_{*}$(coincides with $f_{!}$ when $f$ proper) the pushforward provided by the six-functor formalism.
\end{lemma}

\begin{proof}
Recall that for a smooth $k$-scheme $X$, we denote by $M(X)$ its associated motive 
in either $DM_{gm}(k,\Q)$ or $D_{\mathbb{A}^{1}}(S,\mathcal{E})$. We may prove both statements simutaneously since $DM_{gm}(k,\Q)$  satisfies the six-functor formalism \cite[Theorem 11.4.5]{MR3971240}.

Each element $g\in G$ induces an automorphism $g:X_L\to X_L$, hence an 
endomorphism $g^{*}$ of $M(X_L)$ and, by functoriality, an endomorphism of 
$f_*f^{*}M(X)=f_{*}M(X_{L})$.

Since $f$ is a finite Galois covering morphism and $X$ is smooth, Galois descent for motives \cite[Corollary~3.3.39]{MR3971240}  now yields a canonical isomorphism
\[
M(X) \;\xrightarrow{\ \sim\ }\; (f_*f^*M(X))^G
\]

Since $|G|$ is invertible in the field of coefficients of the underlying category, $(f_{*}f^{*}M(X))^{G}$ is given explicitly by the projector 
\[
p:=\frac{1}{|G|}\sum_{g\in G} g
\]
where $g$ denotes the Galois action on $f_{*}f^{*}M(X)$ induced by $g: X_{L} \to X_{L}$.

With the identifications, we pass to cohomology groups and get
\[
f_*f^*(\alpha)=|G|\cdot \alpha,
\] for any $\alpha\in H^{i}(X,-)$
as claimed.
\end{proof}

Now we are ready to prove the main result.

\begin{proof}[Proof of Theorem~\ref{t5}]
Consider the Cartesian diagram
\[
\begin{CD}
X_{L}\times_{X}X_{L} @>{\mathrm{pr}_{2}}>> X_{L} \\
@V{\mathrm{pr}_{1}}VV                      @VV{f}V \\
X_{L}              @>{f}>>               X
\end{CD}
\]

By the exchange isomorphism in the six-functor formalism, we obtain
\[
f^{*}f_{*} \cong \mathrm{pr}_{1*}\mathrm{pr}_{2}^{*}.
\]
Since $f$ is a finite Galois covering with group $G$, there is a natural
isomorphism
\[
X_{L}\times_{X}X_{L}\cong \bigsqcup_{g\in G} X_{L}.
\]
Under this identification, the functor $\mathrm{pr}_{1*}\mathrm{pr}_{2}^{*}$
corresponds to the direct sum of the pullbacks $g^{*}$.
Hence
\[
f^{*}f_{*} = \sum_{g\in G} g^{*}.
\]

Now let $\beta \in H^{p}(X_L,-)^{G}$. Then
\[
f^{*}f_{*}(\beta)
= \sum_{g\in G} g^{*}(\beta)
= |G|\cdot \beta,
\]
since $\beta$ is $G$-invariant. As $|G|$ is invertible in the coefficient
field, we obtain
\[
\beta = \frac{1}{|G|} f^{*}f_{*}(\beta),
\]
which shows that $f^{*}$ is surjective onto the invariant part.

Injectivity follows from Lemma~\ref{trr}, which asserts that
$f_{*}f^{*} = |G|\cdot \mathrm{id}$ in the level of cohomology groups.
Therefore,
\[
H^{p}(X_L,-)^{G} \cong H^{p}(X,-),
\]
as claimed.
\end{proof}

\subsection{Generalized compatibility}
To conclude the paper, we will construct the Weil restriction map for both motivic cohomology and mixed Weil cohomology in light of \S 3.1. 
We then prove a generalized compatibility theorem.

We fix a perfect field $k$ and let $K$ be a field of characteristic zero. 
Let $E$ be a mixed Weil theory and let $\mathcal{E}$ denote its associated 
commutative ring spectrum. Recall that $DM_{gm}(k,\Q)$ denotes Voevodsky's 
triangulated category of geometric motives with rational coefficients. 
We write $D_{\mathbb{A}^{1}}(k,\mathcal{E})$ for the triangulated category 
of $\mathcal{E}$-modules over $S=\Spec(k)$ with coefficients in $K$.

To simplify notation, for a smooth $k$-scheme $Z$ and an integer 
$i \in \Z$, we write $H^{i}(Z)(p)$ for the corresponding cohomology groups with $p$th Tate twist,
computed either in $DM_{gm}(k,\Q)$ or in 
$D_{\mathbb{A}^{1}}(k,\mathcal{E})$, according to the context.

Now we begin our construction. Let $L/k$ be a finite Galois extension 
with Galois group $G := \Gal(L/k)$, and let $X$ be a smooth 
quasi-projective scheme over $L$. 

Recall that the Weil restriction 
$\mathcal{R}(X)$ of $X$ exists and is a smooth quasi-projective 
scheme over $k$. Moreover, there is an isomorphism
\[
\mathcal{R}(X)_{L} \cong \prod_{\sigma \in G} X^{\sigma},
\]
where $X^{\sigma}$ denotes the conjugate $L$-scheme obtained from $X$ 
via $\sigma \in G$ (see \S 1.1).

Let $\alpha \in H^{i}(X)(p)$ be a cohomology class. 
For each $\sigma \in G$, denote by 
\[
\alpha^{\sigma} \in H^{i}(X^{\sigma})(p)
\]
the class obtained by pulling back $\alpha$ along the canonical 
isomorphism $\sigma \colon X^{\sigma} \to X$.

We suppose that $n=|G|$. As a result, starting with $\alpha \in H^{i}(X)(p)$, we obtain for each 
$\sigma \in G$ a class $\alpha^{\sigma} \in H^{i}(X^{\sigma})(p)$. 
Taking their external (Künneth) product yields a cohomology class in $H^{ni}(\prod_{\sigma\in G}X^{\sigma})(np)$ and
we shall also denote this class by $\prod_{\sigma \in G} \alpha^{\sigma}$. 
To make it clear, we have used the Künneth product homomorphism
\[
\prod_{\sigma \in G} H^{i}(X^{\sigma})(p)
\longrightarrow
H^{i|G|}(\prod_{\sigma \in G} X^{\sigma})(np).
\]

Note that the cohomology class $\prod_{\sigma \in G} \alpha^{\sigma}$
is invariant under the natural $G$-action on 
\[
H^{*}\!\left(\prod_{\sigma \in G} X^{\sigma}\right)(np),
\]
where the action is induced by permuting the factors 
$\{X^{\sigma}\}_{\sigma \in G}$.

As a consequence of Theorem~\ref{t5} and the canonical isomorphism $\mathcal{R}(X)_{L} \cong \prod_{\sigma \in G} X^{\sigma}$,
there exists a unique cohomology class, denoted by $\mathcal{R}(\alpha) \in H^{ni}(\mathcal{R}(X))(np)$, 
such that
\[
f^{*}\mathcal{R}(\alpha)
=
\prod_{\sigma \in G} \alpha^{\sigma},
\]
where 
\[
f \colon \mathcal{R}(X)_{L} \to \mathcal{R}(X)
\]
is the natural projection morphism.

Combining our discussion above, we obtain the Weil restriction map \[\mathcal{R}:H^{*}(X)(p)\to H^{n*}(\mathcal{R}(X))(np)\]

Now we may state and prove our main compatibility theorem.

\begin{thm}
\label{ttttt}
    Let $L/k$ be a finite Galois extension of degree $n$ and $X$ be a smooth quasi-projective scheme over $L$. For any integer $p,q$, we have the following commutative diagram: \[
\begin{CD}
H^{q}(X,\Q(p)) @>{cl^{q}}>> H^{q}(X,\mathcal{E})(p) \\
@V{\mathcal{R}}VV @VV{\mathcal{R}}V \\
H^{nq}(X,\Q(np)) @>{cl^{nq}}>> H^{nq}(X,\mathcal{E})(np)
\end{CD}
\]where vertical maps are the Weil restriction maps we have constructed before and the horizontal maps are the motivic cycle class maps, given in Theorem \ref{tt7}.
\end{thm}

\begin{proof}
    From our construction of the Weil restriction map, it suffices to prove the commutativity of the following diagram \[
\begin{CD}
H^{q}(X,\Q(p)) 
    @>{cl^{q}}>> 
H^{q}(X,\mathcal{E})(p) \\
@V{\prod_{\sigma\in G}\sigma}VV 
@VV{\prod_{\sigma\in G}\sigma}V \\
H^{nq}\!\left(\prod_{\sigma \in G}X^{\sigma},\Q(np)\right) 
    @>{cl^{nq}}>> 
H^{nq}\!\left(\prod_{\sigma \in G}X^{\sigma},\mathcal{E}\right)(np)
\end{CD}\]

However, the motivic cycle class map above is induced by the symmetric monoidal triangulated functor
\[R_\mathcal{E} : DM_{gm}(k,\Q) \longrightarrow D_{\mathbb{A}^{1}}(S,\mathcal{E}).\] (\cite[Theorem 2.7.4]{MR2900540})

In particular,  this functor is compatible with the Künneth product homomorphism in the level of cohomology groups. This concludes our proof.\end{proof}

\begin{rem}
We emphasize that the proof relies solely on the 
Galois descent result (Theorem~\ref{t5}) and Theorem~\ref{tt7}. 
A careful inspection shows that these results follow from 
the model category structure and the six-functor formalism of the 
underlying categories, both of which are intrinsic to their construction.
\end{rem}

We conclude the paper by recalling that, in the usual situation 
(see \cite[\S 3.3]{MR2900540}), the $\ell$-adic cohomology theory 
considered in \S 3 is in fact represented by a mixed Weil theory 
$E_{\mathrm{\acute{e}t},\ell}$ constructed by F.~Déglise 
(\cite[Theorem~3.3.5]{MR2900540}). 
Consequently, Theorem~\ref{ttttt} may be viewed as an extension 
of Theorem~\ref{tt}.

\bibliographystyle{acm}
\bibliography{Guangzhao}

@book {MR1045822,
    AUTHOR = {Bosch, Siegfried and L\"utkebohmert, Werner and Raynaud,
              Michel},
     TITLE = {N\'eron models},
    SERIES = {Ergebnisse der Mathematik und ihrer Grenzgebiete (3) [Results
              in Mathematics and Related Areas (3)]},
    VOLUME = {21},
 PUBLISHER = {Springer-Verlag, Berlin},
      YEAR = {1990},
     PAGES = {x+325},
      ISBN = {3-540-50587-3},
   MRCLASS = {14K15 (11G10 14L15)},
  MRNUMBER = {1045822},
MRREVIEWER = {James\ Milne},
       DOI = {10.1007/978-3-642-51438-8},
       URL = {https://doi-org.login.ezproxy.library.ualberta.ca/10.1007/978-3-642-51438-8},
}

@incollection {MR3727436,
    AUTHOR = {Grothendieck, A. and Deligne, P.},
     TITLE = {La classe de cohomologie associ\'ee \`a{} un cycle},
 BOOKTITLE = {Cohomologie \'etale},
    SERIES = {Lecture Notes in Math.},
    VOLUME = {569},
     PAGES = {129--153},
 PUBLISHER = {Springer, Berlin},
      YEAR = {1977},
      ISBN = {3-540-08066-X; 0-387-08066-X},
   MRCLASS = {14F20 (14C17)},
  MRNUMBER = {3727436},
}

@article {MR4949893,
    AUTHOR = {Karpenko, Nikita and Zhu, Guangzhao},
     TITLE = {Pullback and {W}eil transfer on {C}how groups},
   JOURNAL = {Indag. Math. (N.S.)},
  FJOURNAL = {Koninklijke Nederlandse Akademie van Wetenschappen.
              Indagationes Mathematicae. New Series},
    VOLUME = {36},
      YEAR = {2025},
    NUMBER = {5},
     PAGES = {1476--1480},
      ISSN = {0019-3577,1872-6100},
   MRCLASS = {14C25 (14C15)},
  MRNUMBER = {4949893},
       DOI = {10.1016/j.indag.2025.05.012},
       URL = {https://doi-org.login.ezproxy.library.ualberta.ca/10.1016/j.indag.2025.05.012},
}

@book {MR4904233,
    AUTHOR = {Milne, J. S.},
     TITLE = {\'Etale cohomology},
    SERIES = {Princeton Mathematical Series},
    VOLUME = {33},
      NOTE = {Reprint of [0559531]},
 PUBLISHER = {Princeton University Press, Princeton, NJ},
      YEAR = {2025] \copyright 1980},
     PAGES = {xiii+323},
      ISBN = {9780691273792; 9780691273785; 9780691273778},
   MRCLASS = {14-02 (14F20 18F99)},
  MRNUMBER = {4904233},
}

@article {MR1807268,
    AUTHOR = {Geisser, Thomas and Levine, Marc},
     TITLE = {The {B}loch-{K}ato conjecture and a theorem of
              {S}uslin-{V}oevodsky},
   JOURNAL = {J. Reine Angew. Math.},
  FJOURNAL = {Journal f\"ur die Reine und Angewandte Mathematik. [Crelle's
              Journal]},
    VOLUME = {530},
      YEAR = {2001},
     PAGES = {55--103},
      ISSN = {0075-4102,1435-5345},
   MRCLASS = {14F42 (14C35 19E15)},
  MRNUMBER = {1807268},
MRREVIEWER = {R.\ T.\ Hoobler},
       DOI = {10.1515/crll.2001.006},
       URL = {https://doi-org.login.ezproxy.library.ualberta.ca/10.1515/crll.2001.006},
}

@article {MR1809664,
    AUTHOR = {Karpenko, Nikita A.},
     TITLE = {Weil transfer of algebraic cycles},
   JOURNAL = {Indag. Math. (N.S.)},
  FJOURNAL = {Koninklijke Nederlandse Akademie van Wetenschappen.
              Indagationes Mathematicae. New Series},
    VOLUME = {11},
      YEAR = {2000},
    NUMBER = {1},
     PAGES = {73--86},
      ISSN = {0019-3577},
   MRCLASS = {14C25 (14C15)},
  MRNUMBER = {MR1809664 (2001j:14008)},
MRREVIEWER = {Hideyasu Sumihiro},
}

@incollection {MR4480537,
    AUTHOR = {Ji, Lena and Li, Shizhang and McFaddin, Patrick and Moore,
              Drew and Stevenson, Matthew},
     TITLE = {Weil restriction for schemes and beyond},
 BOOKTITLE = {Stacks {P}roject {E}xpository {C}ollection ({SPEC})},
    SERIES = {London Math. Soc. Lecture Note Ser.},
    VOLUME = {480},
     PAGES = {194--221},
 PUBLISHER = {Cambridge Univ. Press, Cambridge},
      YEAR = {2022},
      ISBN = {978-1-009-05485-0},
   MRCLASS = {14A20},
  MRNUMBER = {4480537},
MRREVIEWER = {Hao\ Sun},
}

@book {MR1644323,
    AUTHOR = {Fulton, William},
     TITLE = {Intersection theory},
    SERIES = {Ergebnisse der Mathematik und ihrer Grenzgebiete. 3. Folge. A
              Series of Modern Surveys in Mathematics [Results in
              Mathematics and Related Areas. 3rd Series. A Series of Modern
              Surveys in Mathematics]},
    VOLUME = {2},
   EDITION = {Second},
 PUBLISHER = {Springer-Verlag, Berlin},
      YEAR = {1998},
     PAGES = {xiv+470},
      ISBN = {3-540-62046-X; 0-387-98549-2},
   MRCLASS = {14C17 (14-02)},
  MRNUMBER = {1644323},
       DOI = {10.1007/978-1-4612-1700-8},
       URL = {https://doi-org.login.ezproxy.library.ualberta.ca/10.1007/978-1-4612-1700-8},
}

@book {MR2242284,
    AUTHOR = {Mazza, Carlo and Voevodsky, Vladimir and Weibel, Charles},
     TITLE = {Lecture notes on motivic cohomology},
    SERIES = {Clay Mathematics Monographs},
    VOLUME = {2},
 PUBLISHER = {American Mathematical Society, Providence, RI; Clay
              Mathematics Institute, Cambridge, MA},
      YEAR = {2006},
     PAGES = {xiv+216},
      ISBN = {978-0-8218-3847-1; 0-8218-3847-4},
   MRCLASS = {14F42 (19E15)},
  MRNUMBER = {2242284},
MRREVIEWER = {Thomas\ Geisser},
}

@book {MR3727161,
    AUTHOR = {Gille, Philippe and Szamuely, Tam\'as},
     TITLE = {Central simple algebras and {G}alois cohomology},
    SERIES = {Cambridge Studies in Advanced Mathematics},
    VOLUME = {165},
   EDITION = {Second},
 PUBLISHER = {Cambridge University Press, Cambridge},
      YEAR = {2017},
     PAGES = {xi+417},
      ISBN = {978-1-316-60988-0; 978-1-107-15637-1},
   MRCLASS = {16K20 (14C35 14F22 19C30)},
  MRNUMBER = {3727161},
}

@book {MR2286826,
    AUTHOR = {Levine, M. and Morel, F.},
     TITLE = {Algebraic cobordism},
    SERIES = {Springer Monographs in Mathematics},
 PUBLISHER = {Springer, Berlin},
      YEAR = {2007},
     PAGES = {xii+244},
      ISBN = {978-3-540-36822-9; 3-540-36822-1},
   MRCLASS = {14F43 (14C15 14L05)},
  MRNUMBER = {2286826},
MRREVIEWER = {Claudio\ Pedrini},
}

@article {MR2900540,
    AUTHOR = {Cisinski, Denis-Charles and D\'eglise, Fr\'ed\'eric},
     TITLE = {Mixed {W}eil cohomologies},
   JOURNAL = {Adv. Math.},
  FJOURNAL = {Advances in Mathematics},
    VOLUME = {230},
      YEAR = {2012},
    NUMBER = {1},
     PAGES = {55--130},
      ISSN = {0001-8708,1090-2082},
   MRCLASS = {14F43 (14C15 14F42 19E15)},
  MRNUMBER = {2900540},
MRREVIEWER = {Shoji\ Yokura},
       DOI = {10.1016/j.aim.2011.10.021},
       URL = {https://doi-org.login.ezproxy.library.ualberta.ca/10.1016/j.aim.2011.10.021},
}

@book {MR3971240,
    AUTHOR = {Cisinski, Denis-Charles and D\'eglise, Fr\'ed\'eric},
     TITLE = {Triangulated categories of mixed motives},
    SERIES = {Springer Monographs in Mathematics},
 PUBLISHER = {Springer, Cham},
      YEAR = {[2019] \copyright 2019},
     PAGES = {xlii+406},
      ISBN = {978-3-030-33241-9; 978-3-030-33242-6},
   MRCLASS = {14F42 (14C15 14C35 18G80 19D55)},
  MRNUMBER = {3971240},
MRREVIEWER = {Igor\ A.\ Rapinchuk},
       DOI = {10.1007/978-3-030-33242-6},
       URL = {https://doi-org.login.ezproxy.library.ualberta.ca/10.1007/978-3-030-33242-6},
}

\end{document}